%% file: specialdigraph.tex
\RequirePackage[l2tabu, orthodox]{nag} %checks syntax

\documentclass[11pt,a4paper,notitlepage]{scrartcl} %KOMA article style as the standard

\usepackage{mathpazo} % math & rm
\usepackage[scaled]{helvet} % ss
\usepackage{courier} % tt
\normalfont
\usepackage[T1]{fontenc}

\usepackage[utf8]{inputenc} % allows to input ä, ö, etc.

\usepackage[pdftex,bookmarks=true,bookmarksnumbered=true,breaklinks=true, colorlinks=true]{hyperref}	%puts links into the pdf
\usepackage[backend=biber]{biblatex}	% customising bibliography
\bibliography{PoissDi}

\usepackage{graphicx}	% to put in pictures
\usepackage{microtype}	% better spacing

\usepackage[a4paper]{geometry}	%geometry package for easy page customisation

\usepackage{enumerate}

\usepackage{amsmath,amsthm,amssymb} 	% mathematical environments, etc.
\usepackage[ 	% cleveref's cref command recognises environments automatically
capitalise 	% capitalises references
]{cleveref}

\theoremstyle{plain}% default
\newtheorem{thm}{Theorem}\crefname{thm}{Theorem}{Theorems}
\newtheorem{lem}[thm]{Lemma}\crefname{lem}{Lemma}{Lemmas}
\newtheorem{prop}[thm]{Proposition}\crefname{prop}{Proposition}{Propositions}
\newtheorem{cor}[thm]{Corollary}\crefname{cor}{Corollary}{Corollaries}

\theoremstyle{definition}
\crefname{defn}{Definition}{Definitions}
\crefname{conj}{Conjecture}{Conjectures}

\theoremstyle{remark}
\newtheorem{rem}{Remark}\crefname{rem}{Remark}{Remarks}
\crefname{obs}{Observation}{Observations}

\def\1{{\mathchoice {1\mskip-4mu\mathrm l}      % Blackboard bold 1
{1\mskip-4mu\mathrm l}
{1\mskip-4.5mu\mathrm l} {1\mskip-5mu\mathrm l}}}

\newcommand{\ssup}[1] {{\scriptscriptstyle{({#1}})}}

\newcommand{\N}{\mathbb{N}}
\newcommand{\EE}{\mathbb{E}}
\newcommand{\PP}{\mathbb{P}}
\newcommand{\bP}{\mathbf{P}}
\newcommand{\bE}{\mathbf{E}}
\newcommand{\e}{\textrm{e}}
\newcommand{\dint}{\textrm{ d}}
\newcommand{\wi}{W^{\ssup{\textrm{in}}}}
\newcommand{\wo}{W^{\ssup{\textrm{out}}}}
\newcommand{\Gg}{\mathcal{G}}

\newcommand{\Di}{D^{\ssup{\textrm{in}}}}
\newcommand{\Do}{D^{\ssup{\textrm{out}}}}
\newcommand{\vi}{\nu^{\ssup{\textrm{in}}}}
\newcommand{\vo}{\nu^{\ssup{\textrm{out}}}}

\newcommand{\cW}{\mathcal{W}}

\begin{document}

%%----------------------------------------------------------------------------------------------------
%% titlehead

\title{Conditionally Poissonian random digraphs}
\titlehead{Poissonian digraphs}
%\subject{Probability}
\author{Christian M\"onch}

\maketitle

\begin{abstract}
We define a Poissonian model of directed random graphs which generalises the
undirected Poissonian random graph process introduced by Norros and Reittu in
Adv. Appl. Probab. 38 (2006), 59--75. Its loopless simple projection is a
rank-one independent-arc inhomogeneous digraph of the type studied by Cao and
Olvera-Cravioto, Random Struct. Alg. 56 (2020), 722--774. For the Poissonian
multigraph itself, we discuss the relation to Norros-Reittu graphs,
characterise limiting degree distributions, and record explicit exploration
estimates. In particular, we give fixed-depth directed local weak limits,
stopped branching-process couplings with weight-mass collision budgets, a
comparison with the simple projection, and a rare-event concentration
criterion. These estimates are intended as graph-side structural inputs for
later dynamics on the graph.
\par\smallskip

{\footnotesize
}
\end{abstract}

%%----------------------------------------------------------------------------------------------------
%% content

\section{Motivation and positioning}
We study sparse, inhomogeneous random digraphs of large but finite size.
Sparseness means that the number of directed edges, or arcs, is of the same
order as the number of vertices.  Inhomogeneity means that arc probabilities
may depend on both endpoints.  The main regime of interest is scale-free:
the proportion of vertices with in- or out-degree $k$ may decay as
$k^{-\tau}$ for some $\tau>2$.

For undirected inhomogeneous graphs, there is now a large body of structural
and process-level theory; see, for example,
\cite{durrett2007random,van2016random}.  Directed models require more care.
They are natural in applications such as random walks
\cite{bordenave2015random,cooper2012stationary}, PageRank algorithms
\cite{chen2016generalized}, and threshold contact processes
\cite{Chatterjee2016}, but the direction of arcs makes the local exploration
asymmetric and separates forward and backward connectivity.

One standard way to prescribe a sparse directed inhomogeneous graph is the
directed configuration model: fix a random or deterministic in- and
out-degree sequence and choose a graph with that degree sequence uniformly.
This is the graph model in
\cite{bordenave2015random,cooper2012stationary,chen2016generalized,Chatterjee2016},
and structural results for it appear in \cite{cooper2004size}.  Its main
technical cost is that the arcs are not independent.

Independent-arc models avoid this dependence.  A broad inhomogeneous digraph
model was proposed in \cite{bloznelis2012birth}, following the undirected
framework of \cite{bollobas2007phase}.  The later framework of
\cite{cao2020connectivity} covers a general class of loopless independent-arc
inhomogeneous digraphs and proves, among other things, degree and component
results by exploration methods.  The loopless simple projection of the
present model is a rank-one member of that class.  The Poissonian multigraph
itself is a parallel-edge analogue, and here we keep the exploration errors in
an explicit weight-mass form.

This specialization is useful in the heavy-tailed regime.  Bounded-kernel
formulations are not especially convenient for power-law graphs with
$\tau\leq3$, where second moments may be infinite; compare
\cite[Sections 16.4 and 18]{bollobas2007phase}.  The Poissonian
normalisation below accommodates this range directly and allows dependence
between in- and out-weights.  The construction is the directed analogue of
the conditionally Poissonian, or Norros-Reittu, random graph process of
\cite{norros_conditionally_2006}.

\section{Model definition}
The graph is a directed multigraph with a Poissonian number of arcs between
ordered vertex pairs.  It is parametrised by a bivariate probability law
$\lambda_W$ on $(0,\infty)\times(0,\infty)$, called the \emph{weight
distribution}.  We write $W=(\wi,\wo)$ for a generic random variable with
law $\lambda_W$.  The two coordinates represent the asymptotic in- and
out-propensities of a typical vertex.  Sparseness requires matching finite
means,
\[
  \EE \wi=\EE\wo=:\mu<\infty.
\]
We also write
\begin{align*}
  \nu&:=(\vi,\vo)=\bigl(\EE[(\wi)^2],\EE[(\wo)^2]\bigr)
       \in [0,\infty]\times[0,\infty],\\
  \rho&:= \EE(\wi\wo)\leq\sqrt{\vi\vo}.
\end{align*}

Let $\cW=(W_n)_{n\in \N}$ denote a sequence of i.i.d.\
$\lambda_W$-distributed random variables, that is,
$W_n=(\wi_n,\wo_n)\sim W$.  Here and throughout the article `$\sim$'
denotes equality in distribution.  Given $\cW$ and $N\in\N$, we define a
random directed multigraph $\Gg_N(\cW)$ with vertex set
$[N]:=\{1,\dots,N\}$ by inserting for each pair
$(v,w)\in [N]\times[N]$ independently a random number $E(v,w)$ of directed
edges, where
\begin{equation}\label{eq_edgedistr}
E(v,w)\sim\textrm{Poiss}\Big(\frac{\wo_v\wi_w}{L_N}\Big) %\text{ and } \{E(v,w):\; (v,w)\in [N]^2\} \text{ is independent.}
\end{equation}
The normalisation $L_N$ is positive, deterministic or measurable with
respect to $\cW$, satisfying the standing assumption
\begin{equation}\label{eq_LNnorm}
\frac{L_N}{N}\longrightarrow\mu
\qquad\text{almost surely.}
\end{equation}
By the strong law of large numbers this implies the edge-density relation
\begin{equation}\label{eq_LNdef}
\frac{1}{L_N}\sum_{v,w\in[N]}\wo_v\wi_w
=
\big(\mu+o(1)\big)N
\qquad\text{almost surely,}
\end{equation}
which guarantees the sparseness of the graph. We use the common asymptotic
notations $o(\cdot),O(\cdot)$ and $\Theta(\cdot)$. A subscript indicates that
the implicit limit statement holds $\PP$-\emph{asymptotically almost surely},
i.e.\ with probability tending to one under $\PP$ as $N\to\infty$.\footnote{Another
common terminology for this type of statement is `with high
probability'. In this article the reference measure sometimes changes, so we
include it in the notation.}

\begin{rem}
In the undirected Norros-Reittu model, each vertex just gets one weight and
the normalisation $L_N$ equals the sum of all weights. In the directed setting,
there is no straightforward equivalent to this setup. If $|\nu|<\infty$ then
the deterministic choice $L_N=\mu N$ already yields a satisfactory model; its
simple projection is a special case of the model in
\cite{bloznelis2012birth}.
Empirical choices such as $L_N=S_N^{\ssup{\textrm{in}}}$ or
$L_N=S_N^{\ssup{\textrm{out}}}$ also satisfy \eqref{eq_LNnorm}. With a view
toward applications and to include `randomly directed' Norros-Reittu graphs as
a special case of our model, see \cref{sec_NRrelation}, we keep the
normalisation abstract and only impose \eqref{eq_LNnorm}. The weaker relation
\eqref{eq_LNdef} records sparseness, while the degree and exploration limit
theorems below use the stronger normalisation \eqref{eq_LNnorm}.
\end{rem}

The graph is defined conditionally on the weight sequence $\cW$.  We assume
that $\cW$ and the family $(\Gg_N)_{N\in\N}$ are defined on the same
probability space.  The corresponding measure is denoted by $\PP$, and
blackboard font quantities refer to this measure.  Conditional probabilities
and expectations given $\cW$ are written as $\bP_\cW,\bE_\cW$, and so on.
\smallskip

In the following sections we will
\begin{itemize}
\item compare the model with Norros-Reittu graphs and explain why component
statements are treated as external input;
\item calculate the limiting degree law of a uniformly chosen vertex;
\item prove empirical degree convergence;
\item record explicit exploration estimates for typical and growing
neighbourhoods;
\item compare the Poissonian multigraph with its simple projection and state
a rare-event concentration criterion.
\end{itemize}

The general component structure of independent-arc inhomogeneous digraphs is
covered by the framework of \cite{cao2020connectivity}.  Our emphasis is on
the explicit Poissonian rank-one estimates needed as graph-side inputs for
process-level arguments.

\section{Relation to the Norros-Reittu graph process with i.i.d.\ weights}\label{sec_NRrelation}
We now discuss the relation with the original (undirected) Norros-Reittu graph
process $(\text{NR}_N)_{N\in\N}$ and explain how two elementary constructions
carry over to the directed model.  The point is not to rederive the general
component theory.  Rather, the comparison identifies familiar special cases
of $\Gg_N(\cW)$ and fixes the orientation conventions used in the examples
below.  General giant strongly connected component results for
independent-arc inhomogeneous digraphs are treated in
\cite{cao2020connectivity}.
\medskip

To begin with, we give the original definition of $\text{NR}_N$ given in
\cite{norros_conditionally_2006}.  Conditionally on a sequence
$\Lambda=(\Lambda_1,\Lambda_2,\dots)$ of i.i.d.\ capacities, define a random
multigraph $\text{NR}_N(\Lambda)$ by inserting independently for each
unordered pair $\{v,w\}$ of vertices in $[N]$ a random number $E\{v,w\}$ of
edges with
\[
 E\{v,w\}
 =
 \text{Poisson}\Big(\frac{\Lambda_v\Lambda_w}{\bar{L}_N}\Big),
\]
where $\bar{L}_N=\sum_{v\in[N]}\Lambda_v$.  Clearly, the capacities in
$\text{NR}_N$ play the role of the weights in our model.
\medskip

There is a simple exact orientation representation of this construction.
Let \(W_v=(\Lambda_v,\Lambda_v)\) and \(L_N=\sum_{v\in[N]}\Lambda_v\).
Then, conditional on \(\Lambda\), the off-diagonal arcs of
\(\Gg_N(\cW)\) have the same distribution as both of the following models:
\begin{enumerate}[(i)]
\item the sum of two independently generated \(\text{NR}_N(\Lambda)\)
models, with every edge of the first copy oriented from the smaller label to
the larger label and every edge of the second copy oriented in the opposite
direction;
\item a graph \(\text{NR}_N(2\Lambda)\) in which each edge is oriented
independently with probability \(1/2\) in each direction.
\end{enumerate}
Equivalently, if \(W_v=(\Lambda_v/2,\Lambda_v/2)\) and
\(L_N=\frac12\sum_{v\in[N]}\Lambda_v\), then the off-diagonal arcs of
\(\Gg_N(\cW)\) have the same distribution as an independently oriented
\(\text{NR}_N(\Lambda)\) graph.
Loops are excluded from these exact representations.  The directed model with
\(W_v=(\Lambda_v,\Lambda_v)\) has loop intensity \(\Lambda_v^2/L_N\), whereas
the two-NR-copy and \(\text{NR}_N(2\Lambda)\) constructions have loop intensity
\(2\Lambda_v^2/L_N\); similarly, orienting a loop in
\(\text{NR}_N(\Lambda)\) does not split it into two directions.  Thus the full
multigraph laws agree after deleting loops, or after independently thinning
the loops in the Norros-Reittu-side constructions by a factor \(1/2\).

\begin{prop}[Orientations of Norros-Reittu graphs]\label[prop]{prop_highcorrel}
The orientation representations in the preceding paragraph hold for the
off-diagonal edge multiplicities.
\end{prop}
\begin{proof}
For \(u<v\), the directed model with \(W_v=(\Lambda_v,\Lambda_v)\) and
\(L_N=\sum_v\Lambda_v\) has independent arc counts
\[
 E(u,v),E(v,u)\sim
 \operatorname{Poisson}\left(\frac{\Lambda_u\Lambda_v}{L_N}\right).
\]
This is exactly what is obtained from two independent
\(\text{NR}_N(\Lambda)\) graphs after orienting one copy from smaller to
larger labels and the other copy from larger to smaller labels.

For the second representation, the unordered edge multiplicity in
\(\text{NR}_N(2\Lambda)\) has Poisson mean
\[
 \frac{(2\Lambda_u)(2\Lambda_v)}{2L_N}
 =
 2\frac{\Lambda_u\Lambda_v}{L_N}.
\]
Independent orientation with probability \(1/2\) splits this Poisson variable
into two independent Poisson variables of mean
\(\Lambda_u\Lambda_v/L_N\).

The final statement is the same Poisson-splitting calculation with directed
mean
\[
 \frac{(\Lambda_u/2)(\Lambda_v/2)}{\frac12\sum_w\Lambda_w}
 =
 \frac12\frac{\Lambda_u\Lambda_v}{\sum_w\Lambda_w},
\]
which is the mean in each direction after independently orienting
\(\text{NR}_N(\Lambda)\).
\end{proof}
\medskip

Before we conclude this section, we return to the general setting and observe that, when the chosen normalisation is non-decreasing, due to Poisson thinning our model has a dynamical representation as a graph process just like the Norros-Reittu model.
\begin{prop}
Assume in addition that $L_N\le L_{N+1}$. Conditional on $\cW$, $\Gg_{N+1}$ may be obtained from $\Gg_{N}$ by removing each existing edge independently with probability $1-L_N/L_{N+1}$ and adding only vertex $N+1$ together with its incoming and outgoing edges as specified in \eqref{eq_edgedistr}.
\end{prop}
\begin{proof}
This follows immediately upon comparing the edge numbers in both graphs and using the fact that independent thinning with removal probability $p$ of a $\text{Poisson}(\lambda)$-distribution yields a $\text{Poisson}((1-p)\lambda)$-distribution.
\end{proof}

\section{Component structure: external inputs and a warning}
The orientation representations in \cref{prop_highcorrel} are useful for
identifying familiar special cases of the model, but they do not by themselves
transfer component-size statements from the constituent undirected layers to
the directed graph.  We record the basic terminology and then treat
macroscopic component structure as external input.

If there exists a directed path from $v$ to $w$ in $\Gg_N$ we write $v\to w$,
and if there is also a directed path from $w$ to $v$ we write
$v\leftrightarrow w$. Clearly, `$\leftrightarrow $' induces an equivalence
relation on $[N]$ and the equivalence classes, or rather the subgraphs induced
by them, are called \emph{strongly connected components} of $\Gg_N$.  Fix
$v\in[N]$. The subgraph induced by
$\mathcal{F}[v]=\{u\in [N]:v\to u\}$ is called the \emph{forward component} or
\emph{forward cluster} of $v$, and the set
$\mathcal{B}[v]=\{u\in [N]:u\to v\}$ is called the \emph{backward component}
or \emph{backward cluster} of $v$.\footnote{By an abuse of notation, we will
not distinguish between vertex sets and the subgraphs they induce in
$\Gg_N(\cW)$.} $\mathcal{C}[v]=\mathcal{F}[v]\cap\mathcal{B}[v]$ denotes the
strongly connected component of $v$. A weak component is a connected component
of the underlying undirected projection of \(\Gg_N(\cW)\).

One has to be careful with sums or unions of oriented layers.  The weak
projection of a sum of directed graphs is the union of the weak projections of
the layers, and components of the individual layers may merge.  Thus weak
components are not inherited unchanged from the constituent graphs.  For
example, with constant capacities \(\Lambda\equiv c\) and \(1/2<c<1\), two
independent Norros-Reittu layers each have subcritical mean degree \(c\), but
their weak union has asymptotic mean degree \(2c>1\) and therefore has a
giant weak component.

For this reason we do not state component-size corollaries from the oriented
representations above.  The general macroscopic component structure of
independent-arc inhomogeneous digraphs is covered by
\cite{bloznelis2012birth,cao2020connectivity}; critical undirected
Norros-Reittu component asymptotics can be found, for example, in
\cite{van2013critical}.  The remainder of the paper focuses instead on degree
laws and explicit exploration estimates for the rank-one Poissonian model;
these are the parts that are most useful for later process-level arguments.

\section[Vertex degrees in the Poissonian digraph]{Vertex degrees in $\Gg_N(\cW)$}
We use the following conventions for indegree and outdegree of a vertex
$v\in[N]$:
\[
 D_v=(\Di_v,\Do_v),\qquad
 \Di_v:=\sum_{u\in [N]\setminus\{v\}}E(u,v),
 \qquad
 \Do_v:=\sum_{u\in [N]\setminus\{v\}}E(v,u).
\]
Thus loops are not counted towards either directed degree.  When total degree
is mentioned, we count a loop once and write
$\deg(v):=\Di_v+\Do_v+E(v,v)$.  If $v>N$, we set $D_v=(0,0)$.

\begin{prop}[Conditional degrees]\label{prop:conditional-degrees}
For $v\in[N]$, conditional on $\cW$, the random variables $\Di_v$ and
$\Do_v$ are independent and
\[
 \Di_v\sim \operatorname{Poisson}(\Lambda_v^{\ssup{\textrm{in}}}),
 \qquad
 \Do_v\sim \operatorname{Poisson}(\Lambda_v^{\ssup{\textrm{out}}}),
\]
where
\[
 \Lambda_v^{\ssup{\textrm{in}}}
   :=\frac{\wi_v}{L_N}\sum_{u\in[N]\setminus\{v\}}\wo_u,
 \qquad
 \Lambda_v^{\ssup{\textrm{out}}}
   :=\frac{\wo_v}{L_N}\sum_{u\in[N]\setminus\{v\}}\wi_u.
\]
Moreover,
\[
 \deg(v)\sim
 \operatorname{Poisson}\left(
   \Lambda_v^{\ssup{\textrm{in}}}
   +\Lambda_v^{\ssup{\textrm{out}}}
   +\frac{\wo_v\wi_v}{L_N}
 \right).
\]
\end{prop}
\begin{proof}
Under $\bP_\cW$ the edge variables are independent Poisson random variables.
The variables entering $\Di_v$ and $\Do_v$ form two disjoint families, and
the loop variable $E(v,v)$ is disjoint from both.  The assertion follows from
stability of Poisson laws under convolution.
\end{proof}

The only dependence between the degree pairs of finitely many fixed vertices
comes from the finitely many arcs whose two endpoints both lie in that set.
Their total intensity is $O(L_N^{-1})$, and therefore vanishes in the sparse
limit.
\begin{prop}[Finite-dimensional degree convergence]\label{prop:finite-degree-convergence}
For $\PP$-almost every weight sequence $\cW$ and every finite set of distinct
fixed vertices $V=\{v_1,\dots,v_m\}\subset\N$,
\[
 \big(D_{v_1},\dots,D_{v_m}\big)
 \overset{\bP_\cW}{\longrightarrow}
 \big(\widetilde D_{v_1},\dots,\widetilde D_{v_m}\big),
 \qquad N\to\infty,
\]
where the limiting degree pairs are independent and
\[
 \widetilde D_{v_i}\sim
 \operatorname{Poisson}(\wi_{v_i})
 \otimes
 \operatorname{Poisson}(\wo_{v_i}).
\]
\end{prop}
\begin{proof}
It is enough to consider $N\geq \max V$.  Define truncated degree pairs
$\widehat D_{v_i}=(\widehat\Di_{v_i},\widehat\Do_{v_i})$ by deleting all
non-loop arcs with both endpoints in $V$:
\[
 \widehat\Di_{v_i}:=\sum_{u\in[N]\setminus V}E(u,v_i),
 \qquad
 \widehat\Do_{v_i}:=\sum_{u\in[N]\setminus V}E(v_i,u).
\]
The edge-variable families defining
$\widehat D_{v_1},\dots,\widehat D_{v_m}$ are pairwise disjoint, hence these
degree pairs are independent under $\bP_\cW$.  Their conditional means are
\[
 \widehat\Lambda_{v_i}^{\ssup{\textrm{in}}}
 =\frac{\wi_{v_i}}{L_N}\sum_{u\in[N]\setminus V}\wo_u,
 \qquad
 \widehat\Lambda_{v_i}^{\ssup{\textrm{out}}}
 =\frac{\wo_{v_i}}{L_N}\sum_{u\in[N]\setminus V}\wi_u.
\]
By the strong law of large numbers and \eqref{eq_LNnorm},
$\sum_{u=1}^N\wi_u/L_N\to1$ and
$\sum_{u=1}^N\wo_u/L_N\to1$ almost surely.  Since $V$ is fixed, the displayed
means converge to $\wi_{v_i}$ and $\wo_{v_i}$, respectively.  Thus the
truncated degree vector has the stated limit.

It remains to compare the original and truncated vectors.  They differ only
if at least one non-loop arc between two distinct vertices of $V$ is present.
By Markov's inequality,
\[
 \bP_\cW\left(
   (D_{v_1},\dots,D_{v_m})
   \neq
   (\widehat D_{v_1},\dots,\widehat D_{v_m})
 \right)
 \leq
 \sum_{\substack{i,j=1\\ i\neq j}}^m
   \frac{\wo_{v_i}\wi_{v_j}}{L_N}
 \longrightarrow 0.
\]
The conclusion follows.
\end{proof}

If $\rho$ is finite then the model produces a tight, in fact asymptotically
Poisson, number of loops.
\begin{prop}[Number of loops]\label{prop:number-of-loops}
Assume that $\rho<\infty$.  Then the total number of loops
\[
 C_N^\circ:=\sum_{v=1}^N E(v,v)
\]
converges in distribution to $\operatorname{Poisson}(\rho/\mu)$.  The same
convergence holds conditionally under $\bP_\cW$ for $\PP$-almost every weight
sequence $\cW$.
\end{prop}
\begin{proof}
Conditional on $\cW$, the loop variables are independent, with
\[
 E(v,v)\sim \operatorname{Poisson}(\wi_v\wo_v/L_N).
\]
Hence
$C_N^\circ$ is Poisson with parameter
\[
 \Theta_N
 :=\sum_{v=1}^N\frac{\wi_v\wo_v}{L_N}
 =\frac{N}{L_N}\left(\frac1N\sum_{v=1}^N\wi_v\wo_v\right).
\]
By the strong law of large numbers and \eqref{eq_LNnorm},
$\Theta_N\to\rho/\mu$ almost surely.  Since Poisson probabilities are
continuous functions of their parameter, the conditional laws converge to
$\operatorname{Poisson}(\rho/\mu)$; annealed convergence follows by dominated
convergence.
\end{proof}

Finally, the degree of a uniformly chosen vertex has the expected mixed
Poisson law.
\begin{cor}\label{cor:typical-degree}
Let $U_N$ be uniformly chosen from $[N]$, independently of the graph, and put
$D(N)=(\Di_{U_N},\Do_{U_N})$.  Then
\[
 D(N)\longrightarrow D
 \quad\text{in distribution as }N\to\infty,
\]
where, conditionally on a mixing variable $W=(\wi,\wo)$ with law
$\lambda_W$,
\[
 D\sim \operatorname{Poisson}(\wi)\otimes\operatorname{Poisson}(\wo).
\]
\end{cor}
\begin{proof}
Let
\[
 A_N:=\frac{\wi_{U_N}}{L_N}\sum_{u\in[N]\setminus\{U_N\}}\wo_u,
 \qquad
 B_N:=\frac{\wo_{U_N}}{L_N}\sum_{u\in[N]\setminus\{U_N\}}\wi_u.
\]
By \cref{prop:conditional-degrees}, conditional on $\cW$ and $U_N$, the two
coordinates of $D(N)$ are independent Poisson variables with parameters
$A_N$ and $B_N$.  Since
$\sum_{u=1}^N\wi_u/L_N\to1$ and
$\sum_{u=1}^N\wo_u/L_N\to1$ in probability, and since
$(\wi_{U_N},\wo_{U_N})$ has law $\lambda_W$ for every $N$,
\[
 (A_N,B_N)-(\wi_{U_N},\wo_{U_N})\longrightarrow(0,0)
 \quad\text{in probability}.
\]
Here the omitted self-terms $\wi_{U_N}\wo_{U_N}/L_N$ vanish in probability
because $\wi_{U_N}\wo_{U_N}$ is tight while $L_N$ grows linearly.
Thus $(A_N,B_N)$ converges in distribution to $W$.  For each
$k,l\in\N_0$,
\[
 \PP(D(N)=(k,l))
 =
 \EE\left[
   \e^{-A_N}\frac{A_N^k}{k!}
   \e^{-B_N}\frac{B_N^l}{l!}
 \right]
 \longrightarrow
 \EE\left[
   \e^{-\wi}\frac{(\wi)^k}{k!}
   \e^{-\wo}\frac{(\wo)^l}{l!}
 \right],
\]
because the Poisson kernels are bounded continuous functions of the
parameters.  Pointwise convergence of the probabilities on the countable
space $\N_0^2$ proves the claim.
\end{proof}

\section[Empirical degree distribution]{Empirical degree distribution in $\Gg_N(\cW)$}
We turn to the empirical degree distribution. Let $\Di_v,\Do_v$ denote the
in- and outdegree of vertex $v\in[N]$ in the graph $\Gg_N(\cW)$, suppressing
the dependence on $N$ for notational convenience.  Put
\[
 \eta_N(k,l):=\frac{1}{N}\sum_{v=1}^{N}\1\{\Di_v=k,\Do_v=l\},
 \qquad k,l\in\N_0,
\]
and let $\eta$ be the mixed-Poisson law
\[
 \eta(k,l):=
 \EE\left[
 \e^{-\wi}\frac{(\wi)^k}{k!}
 \e^{-\wo}\frac{(\wo)^l}{l!}
 \right].
\]

\begin{thm}[Empirical degree distribution]\label{thm:quencheddegdis}
Under the standing normalisation assumption \eqref{eq_LNnorm},
\[
 d_{\textup{TV}}(\eta_N,\eta)\longrightarrow 0
\]
in $\PP$-probability.  More precisely, for almost every weight sequence
satisfying the conclusions of the strong law and \eqref{eq_LNnorm}, the same
convergence holds in $\bP_\cW$-probability.
\end{thm}

For $m\in\N_0$ write
\[
 \pi_m(x):=\e^{-x}\frac{x^m}{m!},\qquad x\ge0.
\]

\begin{lem}\label{lem:Poissdisbound}
For each $m\in\N_0$ the function $\pi_m$ is bounded and globally Lipschitz on
$[0,\infty)$.  Consequently, for fixed $k,l\in\N_0$, the function
$h_{k,l}(x,y):=\pi_k(x)\pi_l(y)$ is bounded and globally Lipschitz on
$[0,\infty)^2$.
\end{lem}
\begin{proof}
For $m=0$, $\pi_0'(x)=-\e^{-x}$ is bounded.  For $m\ge1$,
\[
 \pi_m'(x)=\e^{-x}\left(\frac{x^{m-1}}{(m-1)!}
        -\frac{x^m}{m!}\right),
\]
which is bounded on $[0,\infty)$.  Hence $\pi_m$ is globally Lipschitz by the
mean value theorem.  The product statement follows because all $\pi_m$ are
bounded by one.
\end{proof}

\begin{lem}\label{lem:degree-mean}
Fix $k,l\in\N_0$ and set
\[
 \bar\eta_N(k,l):=\frac1N\sum_{v=1}^N\pi_k(\wi_v)\pi_l(\wo_v).
\]
For almost every weight sequence satisfying \eqref{eq_LNnorm},
\[
 \bE_\cW\eta_N(k,l)-\bar\eta_N(k,l)\longrightarrow0.
\]
\end{lem}
\begin{proof}
Let
\[
 a_{v,N}:=\wi_v\frac{\sum_{u\ne v}\wo_u}{L_N},
 \qquad
 b_{v,N}:=\wo_v\frac{\sum_{u\ne v}\wi_u}{L_N}.
\]
By the conditional degree law,
\[
 \bE_\cW\eta_N(k,l)
 =
 \frac1N\sum_{v=1}^N\pi_k(a_{v,N})\pi_l(b_{v,N}).
\]
By \cref{lem:Poissdisbound}, it is enough to prove
\[
 \frac1N\sum_{v=1}^N
 \min\{|a_{v,N}-\wi_v|+|b_{v,N}-\wo_v|,1\}
 \longrightarrow0.
\]
The terms coming from the normalisation are bounded by
\[
 \left|\frac{S_N^{\ssup{\textrm{out}}}}{L_N}-1\right|
 \frac1N\sum_{v=1}^N\wi_v
 +
 \left|\frac{S_N^{\ssup{\textrm{in}}}}{L_N}-1\right|
 \frac1N\sum_{v=1}^N\wo_v,
\]
which tends to zero by the strong law and \eqref{eq_LNnorm}.  The loop-exclusion
terms are bounded by
\[
 \frac2N\sum_{v=1}^N
 \min\left\{\frac{\wi_v\wo_v}{L_N},1\right\}.
\]
Since \eqref{eq_LNnorm} holds, this has the same limit as
$2N^{-1}\sum_{v=1}^N\min\{\wi_v\wo_v/N,1\}$.  For fixed $K<\infty$,
\[
 \min\left\{\frac{\wi_v\wo_v}{N},1\right\}
 \le
 \frac{K}{N}+\1\{\wi_v\wo_v>K\}.
\]
Taking the Cesaro limit along the fixed weight sequence and then letting
$K\to\infty$ proves that the loop-exclusion contribution vanishes.
\end{proof}

\begin{lem}\label{lem:degree-variance}
For fixed $k,l\in\N_0$ and almost every weight sequence satisfying
\eqref{eq_LNnorm},
\[
 \operatorname{Var}_{\cW}(\eta_N(k,l))\longrightarrow0.
\]
\end{lem}
\begin{proof}
Let $X_v:=\1\{\Di_v=k,\Do_v=l\}$.  Then
\[
 \operatorname{Var}_{\cW}(\eta_N(k,l))
 =
 \frac1{N^2}\sum_{v=1}^N\operatorname{Var}_{\cW}(X_v)
 +
 \frac2{N^2}\sum_{1\le u<v\le N}
 \operatorname{Cov}_{\cW}(X_u,X_v).
\]
The diagonal contribution is at most $1/N$.  For $u\ne v$, remove the two
edge variables $E(u,v)$ and $E(v,u)$ from the degrees of $u$ and $v$, and let
$X_u'$ and $X_v'$ be the corresponding indicators.  The variables $X_u'$ and
$X_v'$ are independent under $\bP_\cW$.  Moreover $X_u=X_u'$ and $X_v=X_v'$
unless $E(u,v)+E(v,u)>0$.  Therefore
\[
 |\operatorname{Cov}_{\cW}(X_u,X_v)|
 \le
 3\bP_\cW(E(u,v)+E(v,u)>0)
 \le
 3\frac{\wo_u\wi_v+\wo_v\wi_u}{L_N}.
\]
It follows that
\[
 \operatorname{Var}_{\cW}(\eta_N(k,l))
 \le
 \frac1N+
 \frac{6}{N^2L_N}
 \sum_{1\le u<v\le N}(\wo_u\wi_v+\wo_v\wi_u)
 \le
 \frac1N+\frac{6S_N^{\ssup{\textrm{out}}}S_N^{\ssup{\textrm{in}}}}
 {N^2L_N},
\]
which tends to zero.
\end{proof}

\begin{proof}[Proof of \cref{thm:quencheddegdis}]
Fix a weight sequence for which the strong law holds and \eqref{eq_LNnorm}.  By
\cref{lem:degree-mean,lem:degree-variance}, for every fixed $k,l\in\N_0$,
\[
 \eta_N(k,l)-\bar\eta_N(k,l)\longrightarrow0
 \qquad\text{in }\bP_\cW\text{-probability}.
\]
The strong law gives $\bar\eta_N(k,l)\to\eta(k,l)$.  Thus
$\eta_N(k,l)\to\eta(k,l)$ in $\bP_\cW$-probability for each fixed $(k,l)$.

It remains to upgrade pointwise convergence to total variation.  Choose
$R<\infty$ so large that
\[
 \eta\big(\{(k,l):k+l>R\}\big)<\varepsilon.
\]
For the empirical tail, Markov's inequality and the conditional degree means
give
\[
 \bE_\cW\eta_N(\{k+l>R\})
 \le
 \frac1R\frac1N\sum_{v=1}^N
 \bE_\cW(\Di_v+\Do_v)
 \le
 \frac{2S_N^{\ssup{\textrm{out}}}S_N^{\ssup{\textrm{in}}}}
 {RN L_N}.
\]
The last expression is bounded by $3\mu/R$ for all large $N$.  Taking $R$
larger if necessary makes this expectation smaller than $\varepsilon^2$, and
Markov's inequality gives
\[
 \bP_\cW\bigl(\eta_N(\{(k,l):k+l>R\})>\varepsilon\bigr)
 \le \varepsilon
\]
for all large $N$.
On the finite set $\{(k,l):k+l\le R\}$, pointwise convergence implies
\[
 \sum_{k+l\le R}|\eta_N(k,l)-\eta(k,l)|
 \longrightarrow0
 \qquad\text{in }\bP_\cW\text{-probability}.
\]
Combining the finite part with the two tail estimates gives
$d_{\textup{TV}}(\eta_N,\eta)\to0$ in $\bP_\cW$-probability.  Since the set of
weight sequences used above has $\PP$-probability one, the annealed
convergence follows.
\end{proof}
\input{poissonian_graph_inputs}

\printbibliography

\end{document}

%% file: poissonian_graph_inputs.tex
\section{Graph inputs for exploration processes}\label{sec:exploration-inputs}

This section records the graph estimates that are needed when processes on
$\Gg_N(\cW)$ are compared with branching processes.  The general
inhomogeneous-digraph framework and the associated exploration philosophy are
developed in \cite{cao2020connectivity}.  Here we keep the rank-one
Poissonian structure explicit, because the resulting error terms are simple
weight-mass quantities and remain useful when the graph is used only as an
input for a later dynamics argument.

For a finite multiset $A$ with labels in $[N]$ write
\[
 W^{\ssup{\textrm{in}}}(A):=\sum_{v\in A}\wi_v,
 \qquad
 W^{\ssup{\textrm{out}}}(A):=\sum_{v\in A}\wo_v,
\]
where labels are counted with multiplicity.  For ordinary sets this agrees
with the usual sum, and $\operatorname{supp}(A)$ denotes the underlying set of
labels.  When $A$ is already a set, $\operatorname{supp}(A)=A$; intersections
between multisets and sets are understood through this support.  Finally set
\[
 S_N^{\ssup{\textrm{in}}}:=\sum_{v=1}^N\wi_v,
 \qquad
 S_N^{\ssup{\textrm{out}}}:=\sum_{v=1}^N\wo_v,
 \qquad
 S_{2,N}^{\ssup{\textrm{in}}}:=\sum_{v=1}^N(\wi_v)^2.
\]
The standing normalisation assumption \eqref{eq_LNnorm}, together with the
strong law, gives
\[
 \frac{S_N^{\ssup{\textrm{in}}}}{L_N}\to1,
 \qquad
 \frac{S_N^{\ssup{\textrm{out}}}}{L_N}\to1
\]
almost surely; these are the normalisation facts used below.

\subsection{Regular variation estimates}

We use survival-tail exponents.  Thus a non-negative random variable $X$ is
regularly varying with exponent $\alpha$ if
\[
 \PP(X>x)=x^{-\alpha}\ell(x),
\]
where $\ell$ is slowly varying.  If a physics degree distribution is described
by a mass exponent $\tau$, then the corresponding survival exponent is
$\alpha=\tau-1$.

\begin{prop}[Power-law maxima and power sums]\label{prop:rv-power-sums}
Let $X_1,X_2,\ldots$ be i.i.d.\ non-negative random variables satisfying
\[
 \PP(X>x)=x^{-\alpha}\ell(x),\qquad \alpha>0,
\]
and let $M_N=\max_{v\le N}X_v$ and
$S_{p,N}=\sum_{v=1}^N X_v^p$.  Then
\[
 M_N=N^{1/\alpha+o_\PP(1)}.
\]
Moreover, if $p<\alpha$, then $N^{-1}S_{p,N}\to \EE X^p$ almost surely, while
if $p>\alpha$, then
\[
 S_{p,N}=N^{p/\alpha+o_\PP(1)}.
\]
In particular, if $\wi$ has survival exponent $\alpha_{\textrm{in}}\in(1,2)$,
then
\[
 S_{2,N}^{\ssup{\textrm{in}}}
 =
 N^{2/\alpha_{\textrm{in}}+o_\PP(1)}.
\]
\end{prop}

\begin{proof}
For the maximum, fix $\varepsilon>0$.  The union bound gives
\[
 \PP(M_N>N^{1/\alpha+\varepsilon})
 \le
 N\PP(X>N^{1/\alpha+\varepsilon})
 =
 N^{-\alpha\varepsilon+o(1)}\to0.
\]
Similarly,
\[
 \PP(M_N\le N^{1/\alpha-\varepsilon})
 =
 \bigl(1-\PP(X>N^{1/\alpha-\varepsilon})\bigr)^N
 \le
 \exp\{-N^{\alpha\varepsilon+o(1)}\}\to0.
\]

The case $p<\alpha$ follows from the strong law, since then
$\EE X^p<\infty$.  For $p>\alpha$, the lower bound follows from
$S_{p,N}\ge M_N^p$.  For the upper bound, choose $\delta>0$ and put
$b_N=N^{1/\alpha+\delta}$.  The maximum estimate implies
$M_N\le b_N$ with probability tending to one.  On this event,
\[
 S_{p,N}
 =
 \sum_{v=1}^N X_v^p\1\{X_v\le b_N\}.
\]
By Karamata's theorem \cite{bingham1987regular},
\[
 \EE\bigl[X^p\1\{X\le b_N\}\bigr]
 =
 N^{p/\alpha-1+\delta(p-\alpha)+o(1)}.
\]
Markov's inequality therefore gives
$S_{p,N}\le N^{p/\alpha+\varepsilon}$ with high probability after taking
$\delta$ sufficiently small.  Since $\varepsilon$ is arbitrary, the claim
follows.
\end{proof}

\begin{prop}[Truncated power sums]\label{prop:rv-truncated-sums}
Let $X_1,X_2,\ldots$ be i.i.d.\ non-negative random variables satisfying
\[
 \PP(X>x)=x^{-\alpha}\ell(x),\qquad \alpha>0.
\]
For $t>0$ put
\[
 S_{p,N}:=\sum_{v=1}^N X_v^p,
 \qquad
 T_{p,N}(t):=\sum_{v=1}^N X_v^p\1\{X_v\le t\},
 \qquad
 U_{p,N}(t):=\sum_{v=1}^N X_v^p\1\{X_v>t\}.
\]
Let $t_N$ be deterministic with $\log t_N/\log N\to\beta\in(0,\infty)$.
If $p<\alpha$, then
\[
 \frac1N T_{p,N}(t_N)\longrightarrow \EE X^p,
 \qquad
 \frac1N U_{p,N}(t_N)\longrightarrow0
\]
in probability.  If $p>\alpha$, then
\[
 T_{p,N}(t_N)
 \le
 N^{\left(1+\beta(p-\alpha)\right)\wedge(p/\alpha)+o_\PP(1)}.
\]
Moreover, the upper bound is sharp away from the extreme-value boundary:
if $\beta<1/\alpha$, then
\[
 T_{p,N}(t_N)=N^{1+\beta(p-\alpha)+o_\PP(1)}
 \quad\text{and}\quad
 U_{p,N}(t_N)=N^{p/\alpha+o_\PP(1)},
\]
whereas if $\beta>1/\alpha$, then
\[
 T_{p,N}(t_N)=S_{p,N}=N^{p/\alpha+o_\PP(1)}
 \quad\text{and}\quad
 U_{p,N}(t_N)=0
\]
with probability tending to one.
\end{prop}

\begin{proof}
For $p<\alpha$, the strong law gives
$N^{-1}\sum_{v=1}^N X_v^p\to\EE X^p$ almost surely.  Also, for every fixed
$R$ and all large $N$ with $t_N\ge R$,
\[
 \frac1N U_{p,N}(t_N)
 \le
 \frac1N\sum_{v=1}^N X_v^p\1\{X_v>R\}.
\]
Taking the limsup and then $R\to\infty$ proves the claim.

Now assume $p>\alpha$.  The full-sum bound
$T_{p,N}(t_N)\le S_{p,N}=N^{p/\alpha+o_\PP(1)}$ follows from
\cref{prop:rv-power-sums}.  Karamata's theorem gives
\[
 \EE\bigl[X^p\1\{X\le t_N\}\bigr]
 =
 N^{\beta(p-\alpha)+o(1)}.
\]
Markov's inequality therefore gives
$T_{p,N}(t_N)\le N^{1+\beta(p-\alpha)+o_\PP(1)}$, proving the displayed upper
bound.

If $\beta<1/\alpha$, fix $\delta>0$ small enough that
$\beta-\delta>0$.  Since $N^{\beta-\delta}<t_N$ for all large $N$, the
number of indices with $X_v\in(N^{\beta-\delta},t_N]$ has mean
$N^{1-\alpha\beta+\alpha\delta+o(1)}\to\infty$ and is concentrated around
its mean.  Hence
\[
 T_{p,N}(t_N)\ge
 N^{p(\beta-\delta)}
 N^{1-\alpha\beta+\alpha\delta+o_\PP(1)}
 =
 N^{1+\beta(p-\alpha)-\delta(p-\alpha)+o_\PP(1)}.
\]
Letting $\delta\downarrow0$ gives the lower bound.  Since $t_N$ is below the
maximum scale, $M_N>t_N$ with high probability, and therefore
$U_{p,N}(t_N)\ge M_N^p=N^{p/\alpha+o_\PP(1)}$; the matching upper bound is
$U_{p,N}(t_N)\le S_{p,N}$.

If $\beta>1/\alpha$, then $t_N$ is above the maximum scale, so
$M_N\le t_N$ with probability tending to one.  On this event
$T_{p,N}(t_N)=S_{p,N}$ and $U_{p,N}(t_N)=0$.
\end{proof}

\begin{lem}\label{lem:size-biased-tails}
Let $X\ge0$ have finite mean and satisfy
\[
 \PP(X>x)=x^{-\alpha}\ell(x),\qquad \alpha>1.
\]
If $X^\star$ has the $X$-size-biased law
\[
 \PP(X^\star\in dx)=\frac{x}{\EE X}\PP(X\in dx),
\]
then
\[
 \PP(X^\star>x)
 =
 \frac{\EE[X\1\{X>x\}]}{\EE X}
 \sim
 \frac{\alpha}{\alpha-1}\frac{x\PP(X>x)}{\EE X}.
\]
Thus size-biasing lowers the survival exponent from $\alpha$ to
$\alpha-1$.
\end{lem}

\begin{proof}
The identity is the definition of the size-biased law.  The asymptotic form is
Karamata's theorem applied to the tail integral
$\EE[X\1\{X>x\}]=x\PP(X>x)+\int_x^\infty\PP(X>y)\dint y$.
\end{proof}

\begin{prop}[Product weights and finite criticality]\label{prop:rv-product-weights}
Let $W=(\wi,\wo)$ have finite marginal means
$\EE\wi=\EE\wo=\mu\in(0,\infty)$.
\begin{enumerate}
\item If $\rho:=\EE[\wi\wo]<\infty$, then
\[
 \frac1N\sum_{v=1}^N\wi_v\wo_v\longrightarrow\rho
\]
almost surely.
\item If $\wi$ and $\wo$ are independent and have survival exponents
$\alpha_{\textrm{in}},\alpha_{\textrm{out}}>1$, then
$\rho=\mu^2$ and
\[
 \PP(\wi\wo>x)=x^{-\alpha_\times+o(1)},
 \qquad
 \alpha_\times:=\alpha_{\textrm{in}}\wedge\alpha_{\textrm{out}}.
\]
\item If $\wi=\wo=X$ and $X$ has survival exponent $\alpha$, then
$\wi\wo=X^2$ has survival exponent $\alpha/2$.  In particular,
$\rho=\EE X^2$ is infinite throughout the infinite-variance range
$\alpha\in(1,2)$.
\end{enumerate}
Consequently, the independent finite-mean power-law model has a finite
branching critical point, whereas the perfectly correlated
infinite-variance model does not.
\end{prop}

\begin{proof}
The first statement is the strong law of large numbers.  If $\wi$ and $\wo$
are independent, then $\rho=\EE\wi\,\EE\wo=\mu^2$.  Put
$\alpha_\times=\alpha_{\textrm{in}}\wedge\alpha_{\textrm{out}}$.  For every
$s<\alpha_\times$,
\[
 \EE[(\wi\wo)^s]=\EE[(\wi)^s]\EE[(\wo)^s]<\infty,
\]
so Markov's inequality gives $\PP(\wi\wo>x)\le x^{-s+o(1)}$.  Conversely,
if $\alpha_\times=\alpha_{\textrm{in}}$, choose $c>0$ with
$\PP(\wo>c)>0$ and use
\[
 \PP(\wi\wo>x)\ge \PP(\wo>c)\PP(\wi>x/c)
 =x^{-\alpha_{\textrm{in}}+o(1)}.
\]
The case $\alpha_\times=\alpha_{\textrm{out}}$ is identical.  Letting
$s\uparrow\alpha_\times$ proves the product-tail exponent.  Finally,
if $\wi=\wo=X$, then
\[
 \PP(\wi\wo>x)=\PP(X>\sqrt{x})=x^{-\alpha/2+o(1)},
\]
and $\EE X^2=\infty$ when $\alpha<2$.
\end{proof}

\subsection{Fixed-depth local weak limits}

The symbol $\nu=(\vi,\vo)$ is already used in the model definition for the
second-moment vector.  We therefore denote size-biased type laws by
$\lambda_W^{\ssup{\textrm{in}*}}$ and
$\lambda_W^{\ssup{\textrm{out}*}}$.  Define their empirical analogues by
\[
 \widehat\lambda_N^{\ssup{\textrm{in}*}}
 :=
 \frac1{S_N^{\ssup{\textrm{in}}}}
 \sum_{v=1}^N \wi_v\delta_{W_v},
 \qquad
 \widehat\lambda_N^{\ssup{\textrm{out}*}}
 :=
 \frac1{S_N^{\ssup{\textrm{out}}}}
 \sum_{v=1}^N \wo_v\delta_{W_v}.
\]
Their limits are
\[
 \lambda_W^{\ssup{\textrm{in}*}}(dw)
 :=
 \frac{w^{\ssup{\textrm{in}}}}{\mu}\lambda_W(dw),
 \qquad
 \lambda_W^{\ssup{\textrm{out}*}}(dw)
 :=
 \frac{w^{\ssup{\textrm{out}}}}{\mu}\lambda_W(dw).
\]

\begin{lem}\label{lem:weighted-empirical-laws}
For every bounded continuous $f:[0,\infty)^2\to\mathbb R$,
\[
 \int f\,\dint\widehat\lambda_N^{\ssup{\textrm{in}*}}
 =
 \frac{\sum_{v=1}^N\wi_v f(W_v)}
      {\sum_{v=1}^N\wi_v}
 \longrightarrow
 \frac{\EE[\wi f(W)]}{\mu}
\]
almost surely.  The analogous convergence holds for
$\widehat\lambda_N^{\ssup{\textrm{out}*}}$.
\end{lem}

\begin{proof}
Since $f$ is bounded and $\EE\wi<\infty$, the variables $\wi_v f(W_v)$ are
integrable.  The strong law applies to the numerator and denominator.  Taking
the ratio proves the in-size-biased statement; the out-size-biased statement
is identical.
\end{proof}

In this subsection a directed out-neighbourhood means the breadth-first
directed unfolding from the root.  Each outgoing arc from a frontier particle
creates a child particle marked by its endpoint label and type; parallel arcs
and repeated endpoint labels are retained as separate children in the
unfolding.  Thus the object is a rooted marked tree, not the full directed
subgraph induced by all vertices at distance at most $r$.

\begin{thm}[Fixed-depth directed local weak limit]\label{thm:directed-lwl}
Let $U_N$ be uniformly distributed on $[N]$, independently of the graph and
weights.  For each fixed $r\ge0$, the marked breadth-first directed
out-neighbourhood tree of $U_N$ in $\Gg_N(\cW)$ up to distance $r$ converges
in distribution to the first $r$ generations of the following typed
Galton--Watson tree:
\begin{itemize}
\item the root type has law $\lambda_W$;
\item a particle of type $w=(w^{\ssup{\textrm{in}}},w^{\ssup{\textrm{out}}})$
has $\operatorname{Poisson}(w^{\ssup{\textrm{out}}})$ children;
\item child types are independent with law
$\lambda_W^{\ssup{\textrm{in}*}}$.
\end{itemize}
The breadth-first in-neighbourhood tree has the analogous limit with the roles
of in- and out-weights exchanged.  If every directed edge is independently
retained with probability $q\in[0,1]$, the forward offspring distribution becomes
$\operatorname{Poisson}(q w^{\ssup{\textrm{out}}})$.
\end{thm}

\begin{proof}
We give the forward proof.  Conditional on the weights and on a source vertex
$u$, the outgoing arcs from $u$ form a Poisson point process on $[N]$ with
intensity
\[
 \Lambda_u^{\ssup{\textrm{out}}}(\{v\})
 =
 \frac{\wo_u\wi_v}{L_N}.
\]
Consequently the total outgoing multiplicity is
$\operatorname{Poisson}(\wo_uS_N^{\ssup{\textrm{in}}}/L_N)$ and, conditional
on this total, the child labels are independent with law
$\wi_v/S_N^{\ssup{\textrm{in}}}$.

Construct the corresponding finite-$N$ typed branching process in which the
root type is $W_{U_N}$, a type $w$ particle has
$\operatorname{Poisson}(w^{\ssup{\textrm{out}}}S_N^{\ssup{\textrm{in}}}/L_N)$
children, and child types have law
$\widehat\lambda_N^{\ssup{\textrm{in}*}}$.
The ordinary empirical law
\[
 \frac1N\sum_{v=1}^N\delta_{W_v}\Rightarrow\lambda_W
\]
holds almost surely by the strong law applied to bounded continuous test
functions, and hence $W_{U_N}\Rightarrow W$.  Together with
\cref{lem:weighted-empirical-laws} and
$S_N^{\ssup{\textrm{in}}}/L_N\to1$, this implies convergence of the
finite-$N$ typed branching process to the stated limit.  Indeed, one proves
this by induction over the depth: after convergence through generation
$s<r$, only finitely many particles are present with probability arbitrarily
close to one, and their Poisson offspring laws and child-type laws converge
jointly.

It remains to compare the branching construction with the graph
unfolding.  If all labels sampled during the breadth-first construction
are distinct and no sampled label has already been exposed, the two rooted
marked objects coincide.  On the event that at most $K$ labels are sampled,
the probability of a repeated sampled label is bounded by
\[
 \binom{K}{2}\sum_{v=1}^N
 \left(\frac{\wi_v}{S_N^{\ssup{\textrm{in}}}}\right)^2.
\]
This tends to zero almost surely, because finite mean implies
$\max_{v\le N}\wi_v/S_N^{\ssup{\textrm{in}}}\to0$.  The probability that one
of the at most $K$ sampled targets equals the uniform root is bounded by
\[
 K\,\EE\left[
 \left.
 \frac{\wi_{U_N}}{S_N^{\ssup{\textrm{in}}}}
 \right|\cW
 \right]
 =
 \frac{K}{N}.
\]
Thus, before $K$ samples, every possible old-label hit is either a target
repeat or a hit on the root, and the probability of such a collision tends to
zero.  Since the limiting depth-$r$ tree is finite almost surely, first choose
$K$ large and then let $N\to\infty$.  This proves the forward convergence.
The backward statement is the same argument applied to incoming arcs, and the
$q$-thinned statement follows from Poisson thinning.
\end{proof}

\begin{cor}\label{cor:branching-criticality}
For the $q$-thinned forward local limit, the mean offspring number of a
non-root particle is
\[
 m(q)
 =
 q\int w^{\ssup{\textrm{out}}}\,
      \lambda_W^{\ssup{\textrm{in}*}}(dw)
 =
 q\frac{\rho}{\mu},
\]
with the convention that $m(q)=\infty$ for $q>0$ if $\rho=\infty$.  Hence, if
$\rho<\infty$, the formal critical value is $q_*=\mu/\rho$.  It is an
admissible transition point in $(0,1)$ precisely when $\rho>\mu$; if
$\rho=\mu$ the process is critical at $q=1$, while if $\rho<\mu$ it is
subcritical for every $q\in[0,1]$.  If $\rho=\infty$, the threshold is
$q_c=0$.

If $\wi$ and $\wo$ are independent, then $\rho=\mu^2$.  In this case
$q_c=1/\mu$ lies in $(0,1)$ only when $\mu>1$; for $\mu=1$ criticality occurs
at $q=1$, and for $\mu<1$ the process is subcritical for all admissible $q$.
\end{cor}

\begin{proof}
The formula for $m(q)$ is immediate from the type law of non-root particles in
\cref{thm:directed-lwl}.  The remaining assertions only record the restriction
$q\in[0,1]$.
\end{proof}

\subsection{Stopped growing explorations}

We now keep $N$ fixed and work conditionally on the weights.  Start from a
finite set $A_0\subset[N]$ of distinct labels.  The unrestricted ideal
labelled out-exploration has a multiset frontier $Q_t$ and an exposed label
set $R_t$, with $Q_0$ equal to $A_0$ as a multiset and $R_0=A_0$.  Given
$Q_t$, each particle $u\in Q_t$ independently produces
\[
 \operatorname{Poisson}\left(\wo_u
 \frac{S_N^{\ssup{\textrm{in}}}}{L_N}\right)
\]
children, whose labels are sampled independently from
$\wi_v/S_N^{\ssup{\textrm{in}}}$.  Let $\widetilde Q_{t+1}$ be the multiset of
all child labels, and set
\[
 Q_{t+1}:=\widetilde Q_{t+1},
 \qquad
 R_{t+1}:=R_t\cup\operatorname{supp}(Q_{t+1}).
\]
This recursion is used on all outcomes; it is not stopped after a collision.
Define
\[
 \tau_{\textrm{coll}}
 :=
 \inf\{s\ge1:\operatorname{supp}(\widetilde Q_s)\cap R_{s-1}\ne\varnothing
 \textrm{ or }\widetilde Q_s\textrm{ contains a repeated label}\}.
\]
Before this time the labelled branching exploration agrees with the
breadth-first out-exploration of the Poissonian multigraph.

\begin{lem}\label{lem:onestep-collision}
Let
\[
 O_t:=W^{\ssup{\textrm{out}}}(Q_t),
 \qquad
 I_t:=W^{\ssup{\textrm{in}}}(R_t).
\]
Conditional on the weights and on the ideal exploration through generation
$t$,
\[
 \bP_\cW(\widetilde Q_{t+1}\cap R_t\ne\varnothing\mid\mathcal F_t)
 \le
 \frac{O_tI_t}{L_N}
\]
and
\[
 \bP_\cW(\widetilde Q_{t+1}\textrm{ contains a repeated label}
      \mid\mathcal F_t)
 \le
 \frac{O_t^2}{2L_N^2}S_{2,N}^{\ssup{\textrm{in}}}.
\]
\end{lem}

\begin{proof}
Conditional on $Q_t$, the number of generated children with labels in
$B\subset[N]$ is Poisson with mean
$O_tW^{\ssup{\textrm{in}}}(B)/L_N$.  Taking $B=R_t$ and using
$1-\e^{-x}\le x$ gives the first bound.

For repeated labels, let $Z_v$ be the number of generated children with label
$v$.  Conditional on $\mathcal F_t$, the variables $Z_v$ are independent
Poisson variables with means $\lambda_v=O_t\wi_v/L_N$.  Hence
\[
 \bP_\cW(\exists v:Z_v\ge2\mid\mathcal F_t)
 \le
 \frac12\sum_{v=1}^N\lambda_v^2
 =
 \frac{O_t^2}{2L_N^2}S_{2,N}^{\ssup{\textrm{in}}}.
\]
\end{proof}

\begin{prop}[Stopped exploration coupling]\label{prop:stopped-coupling}
For $T\ge1$ define the predictable budget
\[
 B_T
 :=
 \sum_{t=0}^{T-1}
 \left(
 \frac{O_tI_t}{L_N}
 +
 \frac{O_t^2}{2L_N^2}S_{2,N}^{\ssup{\textrm{in}}}
 \right),
\]
where the quantities are evaluated in the unrestricted ideal labelled
exploration just defined.  Then
\[
 \bP_\cW(\tau_{\textrm{coll}}\le T)
 \le
 \bE_\cW B_T.
\]
Moreover, if
\[
 \sigma_N(\varepsilon_N)
 :=
 \inf\{T\ge0:B_{T+1}>\varepsilon_N\},
\]
then
\[
 \bP_\cW(\tau_{\textrm{coll}}\le T\wedge\sigma_N(\varepsilon_N))
 \le
 \varepsilon_N.
\]
\end{prop}

\begin{proof}
Sum the two estimates of \cref{lem:onestep-collision} over generations and
take conditional expectation:
\[
 \bP_\cW(\tau_{\textrm{coll}}\le T)
 \le
 \bE_\cW\left[
 \sum_{t=0}^{T-1}
 \bP_\cW(\tau_{\textrm{coll}}=t+1\mid\mathcal F_t)
 \right]
 \le
 \bE_\cW B_T.
\]
The stopped statement follows by applying the same predictable union bound to
$T\wedge\sigma_N(\varepsilon_N)$; by definition the budget accumulated before
this stopping time is at most $\varepsilon_N$.
\end{proof}

If each generated child is independently retained with probability $q$, the
same proof gives the following thinned version.
\begin{cor}\label{cor:thinned-collision-budget}
Fix $q\in[0,1]$ and define
\[
 B_T^{\ssup q}
 =
 \sum_{t=0}^{T-1}
 \left(
 q\frac{O_tI_t}{L_N}
 +
 q^2\frac{O_t^2}{2L_N^2}S_{2,N}^{\ssup{\textrm{in}}}
 \right).
\]
Here $Q_t,R_t,O_t,I_t$ and $\tau_{\textrm{coll}}$ are defined as above for
the unrestricted retained-child exploration, using only retained child labels.
Then
\[
 \bP_\cW(\tau_{\textrm{coll}}\le T)
 \le
 \bE_\cW B_T^{\ssup q}.
\]
Moreover, with
\[
 \sigma_N^{\ssup q}(\varepsilon_N)
 :=
 \inf\{T\ge0:B_{T+1}^{\ssup q}>\varepsilon_N\},
\]
one has
\[
 \bP_\cW(\tau_{\textrm{coll}}\le T\wedge\sigma_N^{\ssup q}(\varepsilon_N))
 \le
 \varepsilon_N.
\]
In particular, if an event $\mathcal G_N$ for the ideal exploration implies
$B_{T+1}^{\ssup q}\le\varepsilon_N$, then
\[
 \bP_\cW(\tau_{\textrm{coll}}\le T,\ \mathcal G_N)
 \le
 \varepsilon_N.
\]
\end{cor}

\begin{proof}
Retaining each generated child with probability $q$ thins the old-label
Poisson mean by $q$ and the two-child repeated-label bound by $q^2$.  The
proof of \cref{prop:stopped-coupling} then applies verbatim.  On
$\mathcal G_N$ we have $T\le\sigma_N^{\ssup q}(\varepsilon_N)$, which gives
the final display.
\end{proof}

\subsection{Simple projection}

Let $G_N^{\textrm{simp}}$ be the simple projection of the Poissonian
multigraph, i.e.
\[
 A_N(u,v)=\1\{E_N(u,v)\ge1\}.
\]
Then
\[
 \bP_\cW(A_N(u,v)=1)
 =
 1-\exp\left(-\frac{\wo_u\wi_v}{L_N}\right).
\]

\begin{prop}[Projection and linearisation]\label{prop:simple-projection-input}
The stopped coupling in \cref{prop:stopped-coupling} and
\cref{cor:thinned-collision-budget} applies to the directed simple projection
when it is generated from the Poissonian multigraph by declaring an arc
present iff its Poisson multiplicity is positive.  Before the first old-label
or repeated-label collision, the multigraph exploration and the simple
projection exploration have the same discovered label tree.

For every deterministic set of queried ordered pairs
$\mathcal P\subset[N]\times[N]$,
\[
 \sum_{(u,v)\in\mathcal P}
 \left[
 \frac{\wo_u\wi_v}{L_N}
 -
 \left(1-\exp\left(-\frac{\wo_u\wi_v}{L_N}\right)\right)
 \right]
 \le
 \frac12\sum_{(u,v)\in\mathcal P}
 \left(\frac{\wo_u\wi_v}{L_N}\right)^2.
\]
In particular, for a queried source set $Q\subset[N]$, the total
linearisation error over all possible targets is bounded by
\[
 \frac12\sum_{u\in Q}\sum_{v=1}^N
 \left(\frac{\wo_u\wi_v}{L_N}\right)^2
 =
 \frac{S_{2,N}^{\ssup{\textrm{in}}}}{2L_N^2}
 \sum_{u\in Q}(\wo_u)^2.
\]
\end{prop}

\begin{proof}
The first statement is a deterministic consequence of the coupling
$A_N(u,v)=\1\{E_N(u,v)\ge1\}$.  The exploration is tree-like precisely until an
old label or a repeated new label is generated.  Parallel arcs from one source
to one target and arcs from different frontier sources to the same target are
therefore both included in the repeated-label stopping rule.  The
linearisation bound follows from
$0\le x-(1-\e^{-x})\le x^2/2$.
\end{proof}

\begin{cor}\label{cor:simple-projection-lwl}
For each fixed $r$, the marked breadth-first in- and out-neighbourhood trees
of a uniformly chosen root in $G_N^{\textrm{simp}}$ have the same local weak
limits as in \cref{thm:directed-lwl}.
\end{cor}

\begin{proof}
In the breadth-first construction used in the proof of
\cref{thm:directed-lwl}, the multigraph and simple projection differ before
depth $r$ only if a sampled target label is old or repeated.  The fixed-depth
collision estimate in that proof shows that this event has probability
tending to zero.
\end{proof}

\subsection{Two-seed and rare-event inputs}

The following form is useful when one wants a statement that is quenched in
the realised graph rather than only in the weights.  It is still a graph
estimate: the marks or thinning used by a later process enter only through
the retained labelled exploration.

Run two $q$-thinned ideal labelled explorations with distinct roots.  For
$i=1,2$, let $Q_t^{\ssup i}$ and $R_t^{\ssup i}$ be their frontiers and
exposed sets, and put
\[
 O_t^{\ssup i}:=W^{\ssup{\textrm{out}}}(Q_t^{\ssup i}),
 \qquad
 I_t^{\ssup i}:=W^{\ssup{\textrm{in}}}(R_t^{\ssup i}).
\]
Let $\tau_i$ be the internal collision time of exploration $i$ and let
$\tau_\times$ be the first time a label generated by one exploration is
already exposed in the other exploration, or both explorations generate the
same new label.

\begin{prop}[Two-seed cross-collision budget]\label{prop:twoseed-budget}
With
\[
 B_{\times,T}^{\ssup q}
 :=
 \sum_{t=0}^{T-1}
 \left[
 q\frac{O_t^{\ssup1}I_t^{\ssup2}
        +O_t^{\ssup2}I_t^{\ssup1}}{L_N}
 +
 q^2\frac{O_t^{\ssup1}O_t^{\ssup2}}{L_N^2}
 S_{2,N}^{\ssup{\textrm{in}}}
 \right],
\]
where the masses are computed in the two unrestricted retained-child ideal
explorations.  Then
\[
 \bP_\cW(\tau_\times\le T,\ \tau_1>T,\ \tau_2>T)
 \le
 \bE_\cW B_{\times,T}^{\ssup q}.
\]
Moreover, with
\[
 \sigma_{\times,N}(\varepsilon_N)
 :=
 \inf\{T\ge0:B_{\times,T+1}^{\ssup q}>\varepsilon_N\},
\]
one has
\[
 \bP_\cW(\tau_\times\le T\wedge\sigma_{\times,N}(\varepsilon_N),\
          \tau_1>T,\ \tau_2>T)
 \le
 \varepsilon_N.
\]
\end{prop}

\begin{proof}
Before the first internal or cross collision, the two explorations query
disjoint ordered edge variables.  Conditional on the weights and the two
histories through generation $t$, the number of retained children generated
by exploration $i$ with label $v$ is Poisson with mean
\[
 \lambda_{i,v}=q\frac{O_t^{\ssup i}\wi_v}{L_N}.
\]
Thus the probability that exploration $1$ hits the exposed set of exploration
$2$ is at most $qO_t^{\ssup1}I_t^{\ssup2}/L_N$, and the symmetric hit gives
the second old-label term.  For a simultaneous new-label overlap,
independence of the disjoint edge variables gives
\[
 \sum_{v=1}^N\lambda_{1,v}\lambda_{2,v}
 =
 q^2\frac{O_t^{\ssup1}O_t^{\ssup2}}{L_N^2}
 S_{2,N}^{\ssup{\textrm{in}}}.
\]
Summing over $t$ proves the displayed estimate, and the stopped version is
again the predictable union bound.  The same proof applies to the simple
projection generated from the Poissonian multigraph, because
$1-\exp(-x)\le x$ and distinct ordered pairs remain independent under the
Poisson construction.
\end{proof}

\begin{cor}[One-sided rare-event criterion]\label{cor:rare-event-criterion}
Let $\widehat H_N$ be an event determined by an ideal $q$-thinned labelled
exploration through generation $T_N$, and put
$p_N:=\bP_\cW(\widehat H_N)$.  Let $H_N(G_N,U,\omega)$ be the corresponding
event for the same exploration performed in the realised graph, where $U$ is
a uniform seed and $\omega$ denotes the independent thinning marks.  Define
\[
 F_N(G_N):=\bP_{G_N}(H_N(G_N,U,\omega)).
\]
Assume that, under the canonical coupling of the graph exploration and the
ideal labelled exploration,
\[
 \1_{H_N(G_N,U,\omega)}\le \1_{\widehat H_N}
 \qquad\text{almost surely,}
\]
and that the two events agree on $\{\tau_{\textrm{coll}}>T_N\}$.  Assume
further, conditionally on the weights, that $Np_N\to\infty$ and that
\[
 \bP_\cW(\tau_{\textrm{coll}}\le T_N,\widehat H_N)=o(p_N),
\]
while for two independent ideal explorations with roots $U_1,U_2$,
\[
 \bP_\cW(U_1\ne U_2,\ \tau_\times\le T_N,\ \tau_1>T_N,\ \tau_2>T_N,\
        \widehat H_N^{\ssup1}\cap\widehat H_N^{\ssup2})
 =
 o(p_N^2).
\]
Then
\[
 F_N(G_N)=p_N+o_{\bP_\cW}(p_N).
\]
\end{cor}

\begin{proof}
The one-sided coupling gives
\[
 0\le p_N-\bE_\cW F_N(G_N)
 \le
 \bP_\cW(\tau_{\textrm{coll}}\le T_N,\widehat H_N)
 =
 o(p_N).
\]
For the second moment, expose two independent seeds and thinning families on
the same realised graph, and write $H_i,\widehat H_i$ for the graph and ideal
events.  The one-sided coupling gives
\[
 \bE_\cW F_N(G_N)^2
 =
 \bP_\cW(H_1\cap H_2)
 \le
 \bP_\cW(\widehat H_1\cap\widehat H_2)
 =
 p_N^2.
\]
For the reverse comparison, discard the event that the two roots coincide;
its contribution is at most $p_N/N=o(p_N^2)$.  With distinct roots, the two
graph explorations agree with the two independent ideal explorations whenever
there is no internal collision and no cross-collision before generation
$T_N$.  Hence
\[
\begin{aligned}
 p_N^2-\bE_\cW F_N(G_N)^2
 &\le
 \frac{p_N}{N}
 +2p_N\,\bP_\cW(\tau_{\textrm{coll}}\le T_N,\widehat H_N)\\
 &\quad+
 \bP_\cW(U_1\ne U_2,\ \tau_\times\le T_N,\ \tau_1>T_N,\ \tau_2>T_N,\
        \widehat H_N^{\ssup1}\cap\widehat H_N^{\ssup2})\\
 &=o(p_N^2).
\end{aligned}
\]
Thus $\bE_\cW F_N(G_N)^2=p_N^2+o(p_N^2)$.  Since also
$\bE_\cW F_N(G_N)=p_N+o(p_N)$, we have
$\operatorname{Var}_{\cW}(F_N(G_N))=o(p_N^2)$, and Chebyshev's inequality
proves the claim.
\end{proof}

\begin{rem}
For a non-monotone event, the one-sided assumption in
\cref{cor:rare-event-criterion} should be replaced by direct control of the
one-seed symmetric difference
$H_N(G_N,U,\omega)\triangle\widehat H_N$ at order $o(p_N)$ and the analogous
two-seed symmetric difference at order $o(p_N^2)$.
\end{rem}